\documentclass[12pt,reqno]{amsart}
\usepackage[usenames,dvipsnames]{xcolor}
\usepackage{amsmath, amsfonts, epsfig, amssymb, graphics}
\usepackage[colorlinks=true,citecolor=cyan]{hyperref}
\usepackage[mathscr]{euscript}
\usepackage{shuffle}

\numberwithin{equation}{section}

\newcommand{\pref}[1]{{\rm (\ref{#1})}}

\def\dinv{\mathrm{dinv}}
\def\area{\mathrm{area}}

\def\riser{\rho}
\def\co{\bleu{\mathbf{co}}}

\def\comp{\models}

\def\srank{\mathrm{rank}}

\def\defn#1{\bleu{\bf #1}}

\def\Z{\mathbb{Z}}

\def\S{{\mathbb S}}
\def\N{{\mathbb N}}

\def\mbf#1{{\mathbf #1}}
\def\part#1#2{\bleu{\rouge{\{} \{#1\}\rouge{,}\{#2\}\rouge{\}}}}
\newcommand{\scalar}[2]{{\langle#1,#2 \rangle}}

\def\auteur#1{{\sc #1}}
\def\titreref#1{{\em #1}}
\def\vol#1{{\bf #1}}

\newcommand{\Park}{\mathscr{P}}

\def\Dyck#1#2{\mathscr{D}_{#1,#2}}
\def\Bin#1#2{\mathscr{B}_{#1,#2}}
\def\cBin#1#2{\mathscr{C}_{#1,#2}}
\def\LBin#1#2{\mathscr{L}_{#1,#2}}
\def\Cat#1#2{{C}_{#1,#2}}
\def\Primitive#1#2{{C}'_{#1,#2}}

\newdimen\carrelength
\def\jcarre{\jaune{\linethickness{\carrelength}\line(1,0){.85}}}

\def\define#1{\bleu{\bf #1}}

\def\bleu{\textcolor{blue}}
\def\rouge{\textcolor{red}}
\def\jaune{\textcolor{yellow}}
\def\bleupale#1{{\color{SkyBlue} #1}}

\newtheorem{theorem}{\bleu{Theorem}}
\newtheorem{proposition}[theorem]{\bleu{Proposition}}

\newtheorem{lemma}[theorem]{\bleu{Lemma}}

\begin{document} 

\title[Interlaced parking functions]{\bleu{\large Interlaced rectangular parking functions}}
 \author[J.-C. Aval]{Jean-Christophe Aval}
\address{LaBRI, CNRS, Universit\'e de Bordeaux,
351 cours de la Lib\'eration, 33405 Talence, France.}
 \email{aval@labri.fr}
  \author[F.~Bergeron]{Fran\c{c}ois Bergeron}
\address{D\'epartement de Math\'ematiques, UQAM,  C.P. 8888, Succ. Centre-Ville, 
 Montr\'eal,  H3C 3P8, Canada.}
 \date{\today. This work was supported by NSERC-Canada.}
 \email{bergeron.francois@uqam.ca}

\begin{abstract}
The aim of this work is to extend to a general $\S_m\times\S_n$-module context 
the Grossman-Bizley \cite{bizley,grossman} paradigm that allows the enumeration of
Dyck paths in a $m\times n$-rectangle. 
We obtain an explicit formula for the the ``bi-Frobenius'' characteristic of what we call {\em interlaced} rectangular parking functions in an $m\times n$-rectangle.
These are obtained by labelling the $n$ vertical steps of an $m\times n$-Dyck path by the numbers from $1$ to $n$,
together with an independent labelling of its horizontal steps by integers from $1$ to $m$.
Our formula specializes to give the Frobenius characteristic of the $\S_n$-module of $m\times n$-parking functions in the general situation. Hence,
it subsumes the result of  Armstrong-Loehr-Warrington of~\cite{armstalk} which furnishes such a formula for the special case when $m$ and $n$ 
are coprime integers.\end{abstract}

\maketitle
 \parskip=0pt

{ \setcounter{tocdepth}{1}\parskip=0pt\footnotesize \tableofcontents}
\parskip=8pt  


\section*{Introduction}


The purpose of this paper is to extend to the context of parking functions the Grossman-Bizley enumeration formula (see~\cite{bizley,grossman}) for the number of Dyck-like paths in an $m\times n$-rectangle. These are the south-east lattice paths that start from the north-west corner of the rectangle, end at its south-east corner, and stay below the line between these corners. To each such path we associate a family of parking functions, seen as labellings of the vertical steps of the path. 
We therefore consider enumeration problems for the global set of such functions; which is called the set of $(m,n)$-parking functions.
More precisely, we obtain explicit formulas for the character (P\'olya-enumeration) of the $\S_n$-module of $(m,n)$-parking function in the general context. Such formulas have already been established (see \cite{armstalk}) in the special ``coprime'' case, \emph{i.e.}: when $m$ and $n$ are coprime. 

We then extend our approach to get formulas for bi-labelled paths. This means that we independently label south-steps by the number $1$ to $n$, and east-steps by the numbers from $1$ to $m$. We give an explicit formula (see~\ref{eq:theo_xy}) for the character of the resulting $\S_n\times\S_m$-module, thus characterizing its decomposition into irreducibles for the joint (commuting) actions of $\S_n$ and $\S_m$.



\section{Rectangular Dyck paths}\label{sec_dyck}
   
  An \define{$(m,n)$-Dyck paths}  is a south-east lattice path, going from $(0,n)$ to $(m,0)$, which stays above the \define{$(m,n)$-diagonal}. This is the line segment joining $(0,n)$ to $(m,0)$. See Figure~\ref{fig1} for an example. 
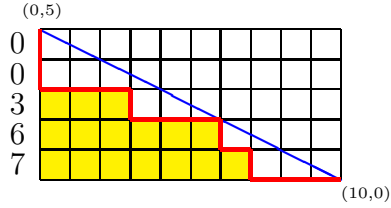
\begin{figure}[ht]
\setlength{\unitlength}{4mm}
\setlength{\carrelength}{4mm}
\def\jcarre{\put(0,0){\jaune{\linethickness{\carrelength}\line(1,0){1}}}}
\def\palecarre{\bleupale{\linethickness{\unitlength}\line(1,0){1}}}
\begin{center}
\begin{picture}(11,5)(0,0)
\put(0,.5){\multiput(0,2)(1,0){3}{\jcarre}
                   \multiput(0,1)(1,0){6}{\jcarre}
                   \multiput(0,0)(1,0){7}{\jcarre}}
\multiput(0,0)(0,1){6}{\line(1,0){10}}
\multiput(0,0)(1,0){11}{\line(0,1){5}}
\thicklines
\put(-.6,5.5){$\scriptscriptstyle(0,5)$}
\put(10,-.6){$\scriptscriptstyle(10,0)$}
 \put(0,5){\bleu{\line(2,-1){10}}}
  \linethickness{.5mm}
\put(0,5){\rouge{\line(0,-1){1}}}\put(-1,0.2){$7$}
\put(0,4){\rouge{\line(0,-1){1}}}\put(-1,1.2){$6$}
\put(0,3){\rouge{\line(1,0){3}}}
\put(3,3){\rouge{\line(0,-1){1}}}\put(-1,2.2){$3$}
\put(3,2){\rouge{\line(1,0){3}}}
\put(6,2){\rouge{\line(0,-1){1}}}\put(-1,3.2){$0$}
\put(6,1){\rouge{\line(1,0){1}}}
\put(7,1){\rouge{\line(0,-1){1}}}\put(-1,4.2){$0$}
\put(7,0){\rouge{\line(1,0){3}}}
\end{picture}\end{center}
\caption{The $(10,5)$-Dyck path encoded as $76300$.}
\label{fig1}
\end{figure}

We encode such paths as decreasing integer sequences
   $$\alpha=a_1a_2\cdots a_n,\qquad {\rm with}\quad 0\leq a_k\leq (n-k)\,m/n.$$
with each $a_k$ giving the distance between the $y$-axis of the (unique) south-step that starts at level $k$.  In other terms, $\alpha$ is some integer partition, with added $0$-parts to make it of length $n$, lying inside the \define{$(m,n)$-staircase} 
   \begin{displaymath}
      \bleu{\delta_{m,n}:=d_1d_2\cdots d_n},\qquad{\rm with}\qquad  \bleu{d_k:= \lfloor (n-k)\,m/n\rfloor}.
   \end{displaymath}
Hence it makes sense to say that the \defn{conjugate path} of an $(m,n)$-path $\alpha$, denoted by $\alpha'$, is the $(n,m)$-path that corresponds to the conjugate partition.   
Examples of staircases are
$$\begin{array}{rclrclrclrcl}
\delta_{{1,4}}=0000,& \delta_{{2,4}}=1100,& \delta_{{3,4}}=2100,& \delta_{{4,4}}=3210,\\[4pt]
\delta_{{5,4}}=3210,& \delta_{{6,4}}=4310,& \delta_{{7,4}}=5310,& \delta_{{8,4}}=6420,\\[4pt]
\delta_{{9,4}}=6420,& \delta_{{10,4}}=7520,& \delta_{{11,4}}=8520,& \delta_{{12,4}}=9630.
\end{array}$$
It is easy to check that $\delta_{kn,n}=\delta_{kn+1,n}$. We denote by $\Dyck{m}{n}$, the set of $(m,n)$-Dyck paths, and by $\Cat{m}{n}$ its cardinality. For example, we have 
 $$\Dyck{5}{4}=\{ 0000, 1000, 2000, 3000, 0011, 2100, 3100, 2200, 3200, 1110, 2110, 3110, 2210, 3210\}.$$
 It follows from the observation that $\delta_{kn,n}=\delta_{kn+1,n}$, that we have the set equality
\begin{equation}\label{observation}
   \bleu{\Dyck{kn}{n}=\Dyck{kn+1}{n}}.
 \end{equation}
Examples of  values of $\Cat{m}{n}:=\#\Dyck{m}{n}$ are given in the following table.
\begin{center}\renewcommand{\arraystretch}{1.5}
\begin{tabular}{c||c|c|c|c|c|c|c|c|c|c|c|}
$m\setminus n$&1&2&3&4&5&6&7&8&9\\
\hline\hline
1&1&1&1&1&1&1&1&1&1\\ 
\hline
2&1&2&2&3&3&4&4&5&5\\ 
\hline
3&1&2&5&5&7&12&12&15&22\\ 
\hline
4&1&3&5&14&14&23&30&55&55\\ 
\hline
5&1&3&7&14&42&42&66&99&143\\ 
\hline
6&1&4&12&23&42&132&132&227&377\\ 
\hline
7&1&4&12&30&66&132&429&429&715\\ 
\hline
\end{tabular}
\end{center}

Observe that it is only when  $k$ takes the form $k=n-j\,b$, with $d=\gcd(m,n)$ and $(m,n)=(ad,bd)$, that we may have $(n-k)\,m/n$ lying in $\N$. Thus $(m,n)$-paths $\alpha$  may only \define{return} to the diagonal at such positions $k$. The set of such return positions clearly forms a subset of $\{1,\ldots,d-1\}$, wich is encoded in the usual manner\footnote{Recall that to a composition $\gamma=(c_1,\ldots,c_k)$ this correspondance associates the set of partial sums $S(\gamma)=\{s_1,s_2,\ldots,s_k\}$, where
    $s_i=c_1+c_2 +\cdots+c_i$, with $1\leq i<k$.} as a composition of $d$.

When $m$ and $n$ are coprime, the enumeration of $(m,n)$-Dyck path is given by the long known formula\footnote{Which may be obtained by a classical cyclic argument maybe due to Dvoretzky-Motzkin (see~\cite{motzkin}), or even earlier to Lukasiewicz.}   
  \begin{displaymath}
   \bleu{\Cat{m}{n} =\frac{1}{m+n}\binom{m+n}{n}}.
\end{displaymath}
Observation~\ref{observation}, and a simple calculation, implies that the classical Catalan numbers (or more generally Fuss-Catalan numbers) can be obtained from this.

When $m$ and $n$ have a greatest common divisor $d$ other than $1$, the relevant formula is more complicated and seems to have escaped the attention of many until recently.
In fact, it appears that it was first stated in 1950 by Grossman, and then proved by Bizley~\cite{bizley} in 1954. As we will see more clearly later, it is usefull to recast this formula in terms of ring homomorphism s applied to symmetric functions. More specifictly, for each fixed coprime pair $(a,b)$, consider the ring homomorphism  
   $$\theta_{a,b}:\Lambda\longrightarrow \Z,$$
such that 
\begin{displaymath}
     	   \bleu{\theta_{a,b}(p_k(\mbf{x})):= \frac{1}{a+b}\binom{ak+bk}{ak}},
\end{displaymath} 
where $p_k(\mbf{x})$ stands for the classical \define{power sum symmetric functions}.
Then, for $(m,n)=(ad,bd)$ with $d=\gcd(m,n)$, Bizley-Grossman formula may be very simply written as
  \begin{equation}\label{bizley_formula}
   \bleu{\Cat{ad}{bd}=\theta_{a,b}(h_d(\mbf{x}))},
\end{equation}
where $h_d(\mbf{x})$ stands\footnote{We are using here Macdonald's~\cite{macdonald} notations.} for the usual \define{complete homogeneous symmetric function}.
Recall that, in generating function format, the link between power sum and complete homogeneous symmetric functions may be expressed as
\begin{equation}
   \bleu{ \sum_{d=0}^\infty h_d(\mbf{x})\,w^d = \exp\!\left(\sum_{k\geq 1} p_k(\mbf{x})\, \frac{w^k}{k}\right)}.
\end{equation} 
Hence, in generating function terms, formula~\ref{bizley_formula}  may be written as
\begin{equation}\label{bizley_gen}
   \bleu{ \sum_{d=0}^\infty \Cat{ad}{bd},w^d = \exp\!\left(\sum_{k\geq 1} \frac{1}{a+b}\binom{ak+bk}{ak}\, \frac{w^k}{k}\right)}.
\end{equation} 
This will be derived from a more general formula in the sequel (see Proposition \ref{prop_bizley}). For example, we have
\begin{eqnarray*}
  \bleu{\Cat{2a}{2b}}&=& \bleu{\frac{1}{2} \left(\frac{1}{a+b} \binom{a+b}{a}\right)^2+\frac{1}{2}\left(\frac{1}{a+b} \binom{2a+2b}{2a}\right)},\\[4pt]
  \bleu{\Cat{3a}{3b}}&=&\bleu{\frac{1}{6}\left(\frac {1}{ a+b} \binom{a+b}{a}\right)^3
  +\frac{1}{2}\left(\frac {1}{a+b} \binom{a+b}{a}\right)\left(\frac{1}{a+b} \binom{2a+2b}{2a}\right)}\\
  &&\qquad \qquad\bleu{+\frac{1}{3}\left(\frac {1}{a+b} \binom{3a+3b}{3a}\right).
}
\end{eqnarray*}
Bizley also showed (and we will see that this generalized as well in the sequel) that the number $\Primitive{ad}{bd}$ of \define{primitive $(ad,bd)$-Dyck paths}, is given by 
    \begin{equation}
        \bleu{\Primitive{ad}{bd}=\theta_{a,b}((-1)^{d-1}e_d(\mbf{x}))}
     \end{equation} 
where the $e_d(\mbf{x})$ are the \define{elementary symmetric functions}. Recall that primitive paths are those that remain strickly above the diagonal. From this, it easily follows that one can enumerate the set of $(m,n)$-Dyck paths with returns to the diagonal encoded by a composition  $\gamma=(c_1,\ldots,c_k)$ of  $d$.
These are the $(m,n)$-Dyck paths that go through the points $(a\,d_i,b\,d_i)$, with 
    \begin{equation}\label{return_points}
        (a\,d_i,b\,d_i),\quad {\rm with}\quad d_i=c_1+\ldots +c_i,\quad {\rm for}\quad 1\leq i\leq k.
     \end{equation}
The relevant enumeration formula is then 
   \begin{equation}
     \bleu{\#\mathscr{Dyck}^{a/ b}_\gamma = \theta_{a,b}((-1)^{d-k}e_{c_1}(\mbf{x})\cdots e_{c_k}(\mbf{x}))},
   \end{equation}
 where we denote by $\mathscr{Dyck}^{a/ b}_\gamma$ the set of $(m,n)$-Dyck paths having returns to the diagonal exactly at the points specified by~\pref{return_points}. 
Clearly, the set $\Dyck{m}{n}$ of all $(m,n)$-Dyck paths decomposes as the disjoint union\footnote{We use summation to denote disjoint union.}
    \begin{displaymath}
    \bleu{\Dyck{m}{n}=\sum_{\gamma\models d} \mathscr{Dyck}_\gamma^{a/ b}},
\end{displaymath} 
where $\gamma\models d$ means that $\gamma$ is a composition of $d$. The set $\mathscr{Dyck}'_{m,n}$ of primitive $(ad,bd)$-Dyck paths simply corresponds to the case of the one part composition $\gamma=(d)$, \emph{i.e.}:
    \begin{displaymath}
    \bleu{\mathscr{Dyck}'_{m,n}= \mathscr{Dyck}_{(d)}^{a/ b}},
\end{displaymath}

\section{\texorpdfstring{$(m,n)$}--parking functions} \label{sec_park}
To each $(m,n)$-Dyck path $\alpha=a_{1} a_{2}\cdots a_{n}$, we associate the set $\Park_\alpha$ of \define{$\alpha$-parking function}:
\begin{displaymath}
   \bleu{\Park_\alpha:=\{a_{\sigma(1)} a_{\sigma(2)}\cdots a_{\sigma(n)} \ |\ \sigma\in\S_n\}}.
    \end{displaymath}
For $\pi\in\Park_\alpha$ we also say that $\alpha$ is the \define{shape} of $\pi$. Observe that $\alpha$-parking functions may be identified with standard Young tableaux\footnote{Naturally using french notation.} of skew shape $(\alpha+1^n)/\alpha$, where $\alpha+1^n$ is the partition having parts $a_k+1$. Indeed, if $k$ sits in row $j$ of $(\alpha+1^n)/\alpha$, then one sets $\pi:=b_1b_2\ldots b_n$, with $b_k:=a_j$.

By definition, the symmetric group acts transitively on $\Park_\alpha$. Indeed, for $\pi=b_1b_2\cdots b_n$, for $\sigma\in\S_n$ on $\pi$, we have
  \begin{displaymath}
   \bleu{\sigma\cdot \pi:=b_{\sigma^{-1}(1)} b_{\sigma^{-1}(2)}\cdots b_{\sigma^{-1}(n)}}.
    \end{displaymath}
 The set of \define{$(m,n)$-parking functions}, denoted by $\Park_{m,n}$,  is the set of $\alpha$-parking functions with $\alpha$ varying in the set of $(m,n)$-Dyck paths:
 $$\bleu{\Park_{m,n}:=\sum_{\alpha\Dyck{m}{n}} \Park_\alpha}.$$
  It clearly affords a permutation action of $\S_n$, whose orbits are the $\Park_\alpha$.
For example, the set $\Park_{3,6}$ contains the $49$ parking functions
\begin{displaymath}\begin {array}{ccccccc} 
000&001&002&003&004&010&011\\ 
012&013&014&020&021&022&023\\ 
024&030&031&032&040&041&042\\  
100&101&102&103&104&110&120\\  
130&140&200&201&202&203&204\\  
210&220&230&240&300&301&302\\  
310&320&400&401&402&410&420.\end {array}
\end{displaymath}
Clearly, the stabilizer of an $(m,n)$-Dyck path $\alpha$ (considered as a  special case of $(m,n)$-parking-function) is the Young subgroup
    $$\S_\rho:=\S_{r_0}\times \S_{r_1}\times \cdots\times \S_{r_m},$$
where $\rho=\rho(\alpha):=(r_0,r_1,\ldots,r_m)$, with $r_i$ equal to the number of occurences of $i$ in $\alpha$. We may as well remove zero parts from $\rho(\alpha)$, since these parts play no role. The result is said to be the \define{riser composition} of $\alpha$. It follows that the number of $\alpha$-parking function is given by the multinomial coefficient
   \begin{equation}
       \bleu{\#\Park_\alpha=  \binom{n}{\riser(\alpha)}:=\frac{n!}{r_0!r_1!\cdots r_k!}},
   \end{equation}
and thus
 \begin{equation}\label{dim_parking}
     \bleu{\#\Park_{m,n}= \sum_{\alpha\in \Dyck{m}{n} }  \binom{n}{\riser(\alpha)}}.
\end{equation}
When $m$ and $n$ are coprime, $(m,n)$-parking functions may be seen to give cannonical coset representatives of the subgroup $H:=u\,\Z$, with $u=(1,1,\ldots,1)$, inside the abelian group $\Z_{m}^{n}$. 
Here, elements of $\Z_{m}^{n}$ correspond to general sequences of length $m$, with entries between $0$ oand $m-1$; wheras $(m,n)$-parking functions correspond to the special case for which such a sequence becomes an $(m,n)$-Dyck when its entries are sorted (from smallest to largest). Indeed, it may be shown that each coset contains a unique $(m,n)$-parking function. It follows that 

\begin{lemma}[Armstrong-Loehr-Warringtion~\cite{armstalk}]
The number of $(m,n)$-parking functions is
\begin{equation}\label{dim_deuxab}
   \bleu{\#\Park_{m,n}=m^{n-1} },
\end{equation}
when $m$ and $n$ are coprime.
\end{lemma}

When $m$ and $n$ are not coprime,  cosets of $H$ will contain up to $d$ ($>0$) elements that are $(m,n)$-parking functions. Judisciously exploiting this fact, one can get an analog of formula~\pref{bizley_formula}. We will not do this, since it actually follows from an even finer result discussed in Section~\ref{sec_frob_park}. 
Table~\ref{tab3} gives small explicit values.
\begin{table}[ht]\renewcommand{\arraystretch}{1.5}
\begin{center}
\begin{tabular}{c||c|c|c|c|c|c|c|c|}
$n\setminus m$&1&2&3&4&5&6&7&8\\
\hline\hline
1&1&1&1&1&1&1&1&1\\
\hline
2&1&3&3&5&5&7&7&9\\ 
\hline
3&1&4&16&16&25&49&49&64\\ 
\hline
4&1&11&27&125&125&243&343&729\\
\hline 
5&1&16&81&256&1296&1296&2401&4096\\ 
\hline
6&1&42&378&1184&3125&16807&16807&35328\\ 
\hline
7&1&64&729&4096&15625&46656&262144&262144\\ 
\hline
\end{tabular}
\end{center}
\medskip
\caption{Number of $(m,n)$-parking functions.}\label{tab3}
\end{table}

\section{Frobenius of the parking function representations} For a fixed integer partition $\rho$ of $n$, consider the transitive permutation action of $\S_n$ on the set of \define{$\rho$-set-partitions} of $\{1,2,\ldots,n\}$. These are the set partitions that have parts size specified by $\rho$. For example, with $\rho=32$, the corresponding set contains the $10$ set partitions
 $$  \begin{matrix}
      \part{1,2,3}{4,5}, & \part{1,2,4}{3,5}, & \part{1,2,5}{3,4}, & \part{1,3,4}{2,5},\\[4pt]
      \part{1,3,5}{2,4},  &  \part{1,4,5}{2,3}, &  \part{2,3,4}{1,5}, &  \part{2,3,5}{1,4},\\[4pt]
      \part{2,4,5}{1,3},   &  \part{3,4,5}{1,2}.
       \end{matrix} $$
Up the linearization, one may consider this as a representation of $\S_n$, having dimension equal to the multinomial coefficient
   \begin{displaymath} \bleu{\binom{n}{\rho} = \frac{n!}{\rho_1!\rho_2!\cdots \rho_k!}}.\end{displaymath}
It is classical that the Frobenius transform\footnote{We simply say: Frobenius characteristic.} of the character of the resulting $\S_n$-module is $h_\rho(\mbf{x}):=h_{\rho_1}(\mbf{x})h_{\rho_2}(\mbf{x})\cdots h_{\rho_k}(\mbf{x})$. Recall that this means that the coefficients of the Schur-function expansion of $h_\rho(\mbf{x})$ correspond to multiplicity of irreducibles.
 
For a given $(m,n)$-Dyck path $\alpha$ of height $n$, the $\S_n$-action on $\alpha$-parking functions  is isomorphic to the action of $\S_n$ on $\rho$-set-partitions, with $\rho=\rho(\alpha)$ equal to the riser-composition $\alpha$.
We write 
    $$\bleu{\alpha(\mbf{x}):=h_{\rho(\alpha)}(\mbf{x})},$$
 for the associated Frobenius characteristic.

 It follows from this that the Frobenius characteristics of the $\S_n$-action on $(m,n)$-parking functions, which we denote by $\Park_{m,n}(\mbf{x})$, can be calculated as follows
  \begin{equation}\label{park_expl}
      \bleu{\Park_{m,n}(\mbf{x})=\sum_{\alpha\in \Dyck{m}{n}}   \alpha(\mbf{x})}.
    \end{equation}
As discussed in~\cite{armstalk}, and borrowing a presentation format inspired by \cite{stanley_parking}, we have the following formulas.
  \begin{proposition}[Armstrong-Loehr-Warrington]\label{prop_frob_park}
When $a$ and $b$ are coprime, we have
   \begin{eqnarray}
     \bleu{\Park_{m,n}(\mbf{x})} 
            &=&\bleu{\frac{1}{m}\,\sum_{\lambda\vdash n} m^{\ell(\lambda)}\, \frac{p_\lambda(\mbf{x})}{z_\lambda} }\label{park_h}\\   
      &=&\bleu{\frac{1}{m}\,\sum_{\lambda\vdash n} s_\lambda(1^m)\,s_\lambda(\mbf{x})} \label{multiciplites_park}\\
      &=&\bleu{ \sum_{\lambda\vdash b} \frac{(m-1)\,(m-2)\cdots (m-\ell(\lambda)+1)}{d_1(\lambda)!\cdots  d_k(\lambda)! }\,h_\lambda(\mbf{x})}\nonumber,
   \end{eqnarray} 
 \bleu{where $d_i(\lambda)$ is the number of parts of size $i$ in $\lambda$}.
\end{proposition}
From~\pref{multiciplites_park}, we may calculate that the respective multiplicities of the trivial and the sign representation in $\Park_{m,n}$, are given (as expected) by
      \begin{displaymath}
        \bleu{\frac{1}{m}\,h_n(1^m)=\frac{1}{m+n}\,\binom{m+n}{n}},
 \qquad {\rm and}\qquad
         \bleu{\frac{1}{m}\,e_n(1^m)=\frac{1}{m}\,\binom{m}{n}},
       \end{displaymath}
 since these occur respectively as coefficients of $s_n(\mbf{x})$ and $s_{1^n}(\mbf{x})$. More generally, the other multiplicities may be obtained using the classical evaluation (involving hook lengths) 
    $$s_\lambda(1^m)=\prod_{(i,j)\in\lambda} \frac{m+i-j}{(\lambda_j-i)+(\lambda'_i-j)+1}.$$
It is also worth recalling that the sign-twisted version\footnote{This simply means that we replace $ p_k(\mbf{x})$ by $(-1)^{k-1}p_k(\mbf{x})$ in \pref{park_h}.} of $\Park_{n,n}(\mbf{x})$ is the Frobenius of the space of ``diagonal harmonics''.

\subsection*{Formula for the non coprime case}\label{sec_frob_park}
To get formulas for the non coprime case, we generalize \pref{bizley_formula} in a natural manner as below. Once again, we assume that $(m,n)=(ad,bd)$ with $(a,b)$ coprime, and consider $\gamma$ a composition of $d$. We adapt to parking functions our notations from Section~\ref{sec_dyck}, hence we denote by $\Park^{a/b}_{\gamma}$ the set of $(ad,bd)$-parking function whose underlying path lies in the set $\mathscr{D}^{a/b}_\gamma$:
   $$\bleu{\Park^{a/b}_{\gamma}:=\sum_{\alpha\in \mathscr{D}^{a/b}_\gamma}\Park_\alpha}.$$
Likewise, $\Park_{m,n}'$ is the set of \define{primitive $(m,n)$-parking functions},  those whose underlying paths only touch the diagonal at both ends. Keeping on with our previous conventions, we set
   $$\bleu{\Park^{a/b}_{\gamma}(\mbf{x}):=\sum_{\alpha\in \mathscr{D}^{a/b}_\gamma}\alpha(\mbf{x})},\qquad {\rm and}\qquad 
   \bleu{\Park'_{m.n}(\mbf{x}):=\sum_{\alpha\in \mathscr{D}'_{m,n}}\alpha(\mbf{x})}.$$
In the same spirit as previously, we consider a ring homomorphism  $\Theta_{a,b}:\Lambda\longrightarrow \Lambda$ that sends degree $d$ homogeneous symmetric functions to degree $n=bd$ homogeneous symmetric functions. Just as before, this homomorphism  is characterized by its effect on the algebraic generators $p_k(\mbf{x})$, setting
\begin{displaymath}
     	   \bleu{\Theta_{a,b}(p_k(\mbf{x})):= \frac{1}{a}\,\sum_{\lambda\vdash n} (ak)^{\ell(\lambda)}\, \frac{p_\lambda(\mbf{x})}{z_\lambda} },
\end{displaymath}
For a degree $d$ symmetric function function $f_d(\mbf{x})$, we also write $f_d^{a/b}(\mbf{x})$ for the image of $f_d(\mbf{x})$ under the homomorphism  $\Theta_{a,b}$, \emph{i.e.}:
    $$\bleu{f_d^{a/b}(\mbf{x}) := \Theta_{a,b}(f_d(\mbf{x}))}.$$
The following proposition extends to the $(m,n)$-parking function Frobenius the approach of Bizley (see  \cite{bizley}) for the enumeration of $(m,n)$-Dyck paths.  Its proof makes use of the notion of \define{rank} of points along $(m,n)$-paths, which is simply defined as
    $$\bleu{\srank(x,y):= m\,y+n\,x}.$$

\begin{proposition}\label{prop_bizley}
Let $(m,n)=(ad,bd)$, with $a$ and $b$ coprime, and consider $\gamma=c_1\cdots c_k$ a composition of $d$. Then we have
 \begin{eqnarray}
     	   \bleu{\Park_{m,n}(\mbf{x})}&=&\bleu{\Theta_{a,b}(h_d(\mbf{x}))},\\
	   \bleu{\Park'_{m,n}(\mbf{x})}&=&\bleu{\Theta_{a,b}((-1)^{d-1}e_d(\mbf{x}))},\qquad {\rm and}\qquad\label{primitif_bizley}\\
	   \bleu{\Park^{a/b}_{\gamma}(\mbf{x})}&=&\bleu{\Theta_{a,b}( (-1)^{d-k}e_{c_1}(\mbf{x})\cdots e_{c_k}(\mbf{x}))}.\label{retours_bizley}
\end{eqnarray} 
\end{proposition}

\begin{proof}[\bf Proof]
The proof is essentially an adaptation of Bizley's original proof, integrating symmetric functions arguments. 

For the purpose of our argument, we consider the set $\Bin{m}{n}$ of all south-east lattices paths going from $(0,n)$ to $(m,0)$, 
which end with an east-step (hence there is no condition relative to the diagonal). We think of these as length $m+n$ ``words'' in the letters $y$ and $x$, which encode the successive steps, with $y$ standing for a south-step and $x$ for an east-step. Thus, the paths in $\Bin{m}{n}$ bijectively correspond to all possible words containing $n$ copies of $y$ and $m$ copies of $x$, with final letter equal to $x$.
It clearly follows that the number of such words/paths is
\begin{equation}
       |\Bin{m}{n}|={m+n-1 \choose n}.
\end{equation}
We bijectively label the $n$ south-steps of such paths with the integers $1$ to $n$,  just as we earlier did for parking functions. 
This is to say that labels decrease along consecutive south-steps. The resulting set of labelled paths is denoted by $\LBin{m}{n}$.
The symmetric group $\S_n$ acts on $\LBin{m}{n}$ by permuting labels,
and the Frobenius characteristic of this (permutation) action is 
\begin{equation}\label{eq:FrobBin}
\sum_{\alpha\in\Bin{m}{n}}\alpha(\mbf{x})=\sum_{\lambda\vdash n} (ak)^{\ell(\lambda)}\, \frac{p_\lambda(\mbf{x})}{z_\lambda} =a\,p_d^{a/ b}(\mbf{x}),
\end{equation}
since the stabilizers of any orbit of $\LBin{m}{n}$ (which corresponds to some
fixed underlying path $\alpha$) is the Young subgroup $\S_{r_1}\times \S_{r_2}\times\cdots\times \S_{r_j}$,
where $\riser(\alpha)=(r_1,r_2,\dots,r_j)$.

We next consider {\em highest rank points} along a path $\alpha$ in $\Bin{m}{n}$. Namely, these are the points $(i,j)$ on the path for which $\srank(i,j)$ reaches its maximal value.
To simplify our discussion, we remove the point $(m,0)$ from those considered. 
Clearly, a path $\alpha$ may have more than one highest rank point, but the number of such points is at most $d$. We denote by $\Bin{m}{n}^t$ (resp. $\Dyck{m}{n}^t$) the subset of $\Bin{m}{n}$
(resp. $\Dyck{m}{n}$) consisting of paths with exactly $t$ highest points.

We then consider {\em cyclic permutations} of $\alpha=\ell_1\cdots \ell_{m+n}$ (where either $\ell_i=x$ or $\ell_i=y$) in the following sense.
Choosing any occurence of $x$ in $\alpha$, say at $\ell_i$, we cut the path after this $x$, and build 
a new word $\beta$ by transposing the two resulting components of $\alpha$:
   $$\alpha=\bleu{\ell_1\cdots \ell_i}\,\rouge{\ell_{i+1}\cdots  \ell_{m+n}}\quad \mapsto\quad  
       \beta=\bleu{\ell_{i+1}\cdots  \ell_{m+n}}\,\rouge{\ell_1\cdots \ell_i}.$$
Observe that the number of highest rank points is invariant
under such cyclic permutations; and that these highest rank points eventually lie on the diagonal after
 any cyclic permutation.
Moreover, the riser-composition of $\alpha$ is cyclicly preserved, hence 
   $$\alpha(\mbf{x})=\beta(\mbf{x}).$$
Figure \ref{fig:bizley} illustrates the notion of highest point (two in this case,
represented by blue dots) and the procedure of cyclic permutation
(the cut  $\ell_i$ appears as a black cross).

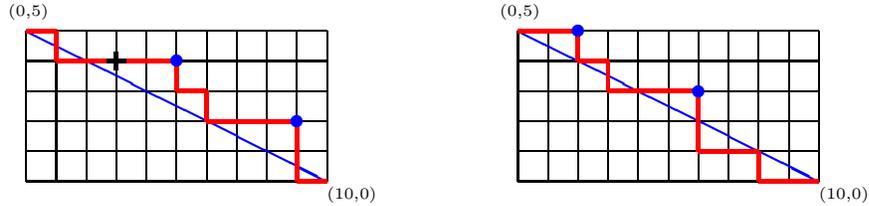
\begin{figure}[ht]
\setlength{\unitlength}{4mm}
\setlength{\carrelength}{4mm}
\def\jcarre{\put(0,0){\jaune{\linethickness{\carrelength}\line(1,0){1}}}}
\def\palecarre{\bleupale{\linethickness{\unitlength}\line(1,0){1}}}
\begin{center}
\begin{picture}(11,5)(0,0)
\multiput(0,0)(0,1){6}{\line(1,0){10}}
\multiput(0,0)(1,0){11}{\line(0,1){5}}
\thicklines
\put(-.6,5.5){$\scriptscriptstyle(0,5)$}
\put(10,-.6){$\scriptscriptstyle(10,0)$}
 \put(0,5){\bleu{\line(2,-1){10}}}
  \linethickness{.5mm}
\put(0,5){\rouge{\line(1,0){1}}}
\put(1,5){\rouge{\line(0,-1){1}}}
\put(1,4){\rouge{\line(1,0){4}}}
\put(5,4){\rouge{\line(0,-1){1}}}
\put(5,3){\rouge{\line(1,0){1}}}
\put(6,3){\rouge{\line(0,-1){1}}}
\put(6,2){\rouge{\line(1,0){3}}}
\put(9,2){\rouge{\line(0,-1){2}}}
\put(9,0){\rouge{\line(1,0){1}}}
\put(4.74,3.74){\bleu{$\bullet$}}
\put(8.74,1.74){\bleu{$\bullet$}}
\put(2.7,4){\line(1,0){0.6}}
\put(3,3.7){\line(0,1){0.6}}
\end{picture}
\hskip 2cm
\begin{picture}(11,5)(0,0)
\multiput(0,0)(0,1){6}{\line(1,0){10}}
\multiput(0,0)(1,0){11}{\line(0,1){5}}
\thicklines
\put(-.6,5.5){$\scriptscriptstyle(0,5)$}
\put(10,-.6){$\scriptscriptstyle(10,0)$}
 \put(0,5){\bleu{\line(2,-1){10}}}
  \linethickness{.5mm}
\put(0,5){\rouge{\line(1,0){2}}}
\put(2,5){\rouge{\line(0,-1){1}}}
\put(2,4){\rouge{\line(1,0){1}}}
\put(3,4){\rouge{\line(0,-1){1}}}
\put(3,3){\rouge{\line(1,0){3}}}
\put(6,3){\rouge{\line(0,-1){2}}}
\put(6,1){\rouge{\line(1,0){2}}}
\put(8,1){\rouge{\line(0,-1){1}}}
\put(8,0){\rouge{\line(1,0){2}}}
\put(1.74,4.74){\bleu{$\bullet$}}
\put(5.74,2.74){\bleu{$\bullet$}}
\end{picture}

\end{center}
\caption{Cyclical permutation of an element of $\Bin{10}{5}$ with 2 highest points.}
\label{fig:bizley}
\end{figure}

The key point is the following: the cyclic permutation allows us
to build a bijection
\begin{equation}\label{bij:bizley1}
\Dyck{m}{n}^t\times [m] \overset{\sim}{\longrightarrow} \Bin{m}{n}^t\times [t]. 
\end{equation} 
Consider $(\alpha,j)\in\Dyck{m}{n}^t\times [m]$.
By cutting $\alpha$ after its $j$-th east step, and performing cyclic permutation,
we get an element $\gamma$ of $\Bin{m}{n}^t$. 
We may keep track of the final point of $\alpha$, which corresponds to one of the $t$
highest points of $\gamma$ (call it $h$).
The reverse bijection consists in performing cyclic permutation to $\gamma$
in position $h$: we get back $\alpha$ and we keep track of the final point 
of $\gamma$ as the index $j$.
  
Now recall from our hypothesis that $(m,n)=(ad,bd)$ with $(a,b)$ coprime.
Denote by $\Park_{ad,bd}^t(\mbf{x})$ the Frobenius characteristic
of the parking functions associated to $(ad,bd)$-Dyck paths having exactly $t$
contact points with the diagonal. 
We get from bijection \eqref{bij:bizley1} that $\sum_{t=1}^d ({da/t})\, \Park_{ad,bd}^t(\mbf{x})$ 
is the Frobenius characteristic of $\LBin{ad}{bd}$.
Whence, because of \eqref{eq:FrobBin} and after simplification:
\begin{equation}\label{eq:bizley}
\sum_{t=1}^d {1\over t}\, \Park_{ad,bd}^t(\mbf{x}) = {1\over d}\,p_d^{a/ b}(\mbf{x}).
\end{equation}
Since the case $t=1$ of $\Park^t_{ad,bd}(\mbf{x})$ corresponds to the Frobenius characteristic primitive $(ad,bd)$-Dyck path, we clearly have
\begin{equation}
\Park_{ad,bd}^t(\mbf{x})=\sum_{\gamma\models_t d} \Park'_{ac_1,bc_1}(\mbf{x})\,\Park'_{ac_2,bc_2}(\mbf{x})\, \cdots\, \Park'_{ac_k,bc_k}(\mbf{x})
\end{equation}
where the sum is over length $t$  compositions $\gamma=({c_1},{c_2},\dots,{c_t})$ of  $d$.
In other terms, if we set 
    $$\Park'_{a,b}(\mbf{x};z):=\sum_{j=1}^\infty \Park'_{aj,bj}(\mbf{x}) \, z^j,$$ 
then  $\Park_{ak,bk}^t(\mbf{x})$ is the coefficient of $z^k$ in ${\big(\Park'_{a,b}(\mbf{x};z)\big)}^t$.
We may thus consider that equation \eqref{eq:bizley} says that ${1\over k}\,p_k^{a/ b}(\mbf{x})$
is the coefficient of $z^k$ in $\log(1/(1-\Park'_{a,b}(\mbf{x};z))$, so that
\begin{eqnarray}\label{eq:Pxz}
\Park'_{a,b}(\mbf{x};z) &=& 1-\exp\Big(-\sum_{k=1}^\infty p_k^{a/ b}(\mbf{x})\,z^k/k\Big)\\
                &=&\Theta_{a,b}\left(\sum_{d=1}^\infty (-1)^{d-1}e_d(\mbf{x})\,z^d\right),
\end{eqnarray}
which is equivalent to \pref{primitif_bizley}, which in turn readily implies \pref{retours_bizley}.  
Clearly, 
  \begin{eqnarray*}
      \Park_{ad,bd}(\mbf{x})&=&\sum_{t>0}\Park_{ak,bk}^t(\mbf{x}),
   \end{eqnarray*}
so that
\begin{eqnarray}\label{eq:Qxz}
\Park_{a,b}(\mbf{x};z)&=&\frac{1}{1-\Park'_{a,b}(\mbf{x};z)}=\exp\Big(\sum_{k=1}^\infty p_k^{a/ b}(\mbf{x})\,z^k/k\Big)\\
      &=& \Theta_{a,b}\left(\sum_{d=0}^\infty h_d(\mbf{x})\,z^d\right),
\end{eqnarray}
which concludes the proof of Proposition \ref{prop_bizley}.
\end{proof}

For example, for any $a$ and $b$ coprime, we get
\begin{eqnarray*}
     	   {\Park_{2a,2b}(\mbf{x})}&=&\frac{1}{2a} h_{2b}[2a\,\mbf{x}] + \frac{1}{2}\left(\frac{1}{a}h_{b}[a\,\mbf{x}]\right)^2,\\
	   {\Park'_{2a,2b}(\mbf{x})}&=&\frac{1}{2a} h_{2b}[2a\,\mbf{x}] - \frac{1}{2}\left(\frac{1}{a}h_{b}[a\,\mbf{x}]\right)^2.
\end{eqnarray*} 
which respectively give the Frobenius that correspond to $(2a,2b)$-Dyck paths, and primitive $(2a,2b)$-Dyck paths. To get explicit formulas for the number of $(m,n)$-parking functions, we need simply compute the scalar product $\scalar{\Park_{m,n}(\mbf{x})}{h_1^n(\mbf{x})}$.
Likewise for $\Park'_{m,n}$ or $\Park^{a/b}_\gamma$.



\section{Bi-Frobenius}

As before, let $(m,n)=(ad,bd)$ with $a$ and $b$ coprime. We can now consider the ``riser-step`` bi-Frobenius of $(m,n)$-parking function, which may be defined/calculated to be 
   \begin{equation}\label{biFrob}
        \bleu{\Park_{m,n}(\mathbf{x},\mathbf{y}):=\sum_{\alpha\in\Dyck{m}{n}} \alpha(\mathbf{x})\,
           \alpha'(\mathbf{y})},
  \end{equation}
where $\mathbf{y}=y_1,y_2,\ldots$ stands for another denumerable alphabet of variables. Hence, $\alpha(\mathbf{x})$ encodes the ``riser structure'' of $\alpha$, whereas $\alpha'(\mathbf{y})$ encodes its ``step structure'' (which are the risers of the conjugate path $\alpha'$). Clearly this bi-Frobenius affords the symmetry
   \begin{equation}
        \bleu{\Park_{m,n}(\mathbf{x},\mathbf{y})=
        \Park_{n,m}(\mathbf{y},\mathbf{x})}.
  \end{equation}

Once again,  there is a Bizley-like formula for $\Park_{m,n}(\mathbf{x},\mathbf{y})$, which subsumes (up to some calculations) all of our previous results. Indeed, we have
 \begin{theorem}\label{theo:bizley_xy}
    For all coprime pair $(a,b)$, the following holds
       \begin{equation}\label{eq:theo_xy}
          \bleu{\sum_{d=0}^\infty \Park_{ad,bd}(\mathbf{x},\mathbf{y})\ z^d =
               \exp\!\left( \sum_{k\geq 1}p_k^{a/b}(\mbf{x},\mbf{y}) z^k/k \right)},
       \end{equation}
  where we set
       \begin{equation}\label{eq:p_xy}
    \bleu{p_k^{a/b}(\mbf{x},\mbf{y}):=\sum_{j=1}^k \frac{k}{j} \sum_{\rho\comp_j bk,\ \sigma\comp_j ak}
            h_\rho(\mathbf{x})\,h_\sigma(\mathbf{y}) },
       \end{equation}
 and we write $\rho\comp_j n$ to say that $\rho$ is a composition of $n$ having $j$ parts.
 \end{theorem}
 
\begin{proof}[\bf Proof]
The proof of Theorem \ref{theo:bizley_xy} uses a refinement of the argument used to prove Proposition \ref{prop_bizley}.
We introduce the set $\cBin{m}{n}$ of lattice paths in $\Bin{m}{n}$ with the additional
condition that they start with a south step.
By definition, {\em corners} of a south-east lattice path correspond to points that lie between a south-step and a following east-step. We consider  $(m,n)$-Dyck paths $\alpha$ having $t$ highest points
and $j$ corners, and modify the argument of Proposition \ref{prop_bizley} by restricting cuts to points that lie at one of the $j$ corners. In the same way as \eqref{bij:bizley1}, we get a bijection
\begin{equation}\label{bij:bizley2}
\Dyck{m}{n}^{t,j}\times [j] \overset{\sim}{\longrightarrow} \cBin{m}{n}^{t,j}\times [t],
\end{equation} 
where superscript $t,j$ indicates a restriction to paths with exactly $t$
highest points and $j$ corners.

Set 
   $$\Park_{m,n}^{t,j}(\mbf{x},\mbf{y}):=\sum_{\alpha\in \mathscr{D}^{j,j}{m,n}} \alpha(\mbf{x})\, \alpha'(\mbf{y})$$ 
Bijection \eqref{bij:bizley2} implies:
$$
{1\over t}\Park_{m,n}^{t,j}(\mbf{x},\mbf{y})={1 \over j}
\sum_{\rho\comp_j n,\ \sigma\comp_j m}
            h_\rho(\mathbf{x})\,h_\sigma(\mathbf{y}).
$$
Summing over $j$ (writing $(m,n)$ as $(ad,bd)$), we obtain
 \begin{equation}\label{eq:bizley_xy}
\sum_{t=1}^d {1\over t}\, \Park_{ad,bd}^t(\mbf{x},\mbf{y}) = 
\sum_{j=1}^d {1 \over j} \sum_{\rho\comp_j bd,\ \sigma\comp_j ad}
            h_\rho(\mathbf{x})\,h_\sigma(\mathbf{y}) 
= {1\over d}\,p_d^{a/b}(\mbf{x},\mbf{y}).
\end{equation}
Then, \eqref{eq:theo_xy} is deduced from \eqref{eq:bizley_xy}, just 
as \eqref{eq:Pxz} was deduced from \eqref{eq:bizley}.

\end{proof}
 An alternate formula for $p_k^{a/b}(\mbf{x},\mbf{y})$, which involves much less terms, is easily seen to be
        \begin{equation}
    \bleu{p_k^{a/b}(\mbf{x},\mbf{y}):=k \sum_{j=1}^k \frac{1}{j} \sum_{\mu\vdash_j bk,\ \nu\vdash_j ak} \binom{j}{\lambda(\mu)}\binom{j}{\lambda(\nu)}
            h_\mu(\mathbf{x})\,h_\nu(\mathbf{y}) }.
       \end{equation}
The second sum is now over $j$-part partitions, with $\lambda(\mu)$ denoting the partition of $j$ which indicates the multiplicities of the parts of $\mu$ (likewise for $\nu$), and $\binom{j}{\lambda}$ stands for the relevant multinomial coefficient. Hence, the above formula is simply obtained by collecting in~\pref{eq:p_xy} the compositions that have same parts structure. 

Observe that we get back Proposition~\ref{prop_bizley} from Theorem~\ref{theo:bizley_xy}, if we take the usual symmetric function scalar product with $h_n(\mbf{y})$ on each side of \pref{eq:bizley_xy}. Indeed, we first recall that this scalar product is such that
  $$\scalar{h_\mu(\mbf{y})}{h_n(\mbf{y})}=1,$$
for any partition (or composition) of $n$. This implies that $\scalar{-}{h_n(\mbf{y})}$ is a ring homomorphism. Hence, we need only prove that
\begin{equation}\label{eq:fin1}
      \scalar{p_k^{a/b}(\mbf{x},\mbf{y})}{h_n(\mbf{y})}=
      k \sum_{j=1}^k \frac{1}{j} \sum_{\mu\vdash_j bk,\ \nu\vdash_j ak} \binom{j}{\lambda(\mu)}\binom{j}{\lambda(\nu)}
            h_\rho(\mathbf{x})
\end{equation}
  is equal to          
\begin{equation}\label{eq:fin2}
        p_k^{a/b}(\mbf{x})=
      \frac{1}{a}\,\sum_{\lambda\vdash bk} (ak)^{\ell(\lambda)}\, \frac{p_\lambda(\mbf{x})}{z_\lambda}.
\end{equation}
This is obtained as follows.
The RHS of \eqref{eq:fin1} is equal to 
$ k \sum_{t,j} \frac{1}{j} \sum_{\gamma\in\cBin{m}{n}^{t,j}} h_{\rho(\gamma)}(\mbf{x})$
whereas the RHS of \eqref{eq:fin2} is equal to 
$ {k\over m} \sum_{\gamma\in\Bin{m}{n}} h_{\rho(\gamma)}(\mbf{x}).$
Because of bijections \eqref{bij:bizley1} and \eqref{bij:bizley2}, these two expressions are equal
(and equal to $ k \sum_{t}  {1\over t} \sum_{\alpha\in\Dyck{m}{n}^t} h_{\rho(\gamma)}(\mbf{x})$).

\section{Further considerations}
Extensions of these considerations, linked to several interesting questions (see~\cite{athanasiadis,ABG,hall,HHLRU,zabrocki}), take into account parameters on parking functions such as ``area'' and ``dinv''. To formulate the analogous results, one needs to work with an algebra of operators on symmetric functions isomorphic to the elliptic Hall algebra studied in \cite{burban,gorsky,schiffmann}. In this framework, the homomorphism $\Theta_{a,b}$ sends a symmetric function to an operator on symmetric functions. In turn, formulas are obtained by applying the resulting operator to the symmetric function $1$.

In this light, it is worth observing that the image under $\Theta_{a,b}$ of other symmetric functions gives rise to significant formulas. The interesting feature of these formulas is that their Schur function expansion  have positive integer coefficients. It is usual to say that they are ``Schur-positive''.  This is the case for hook Schur functions\footnote{We use here the Frobenius notation, hence the relevant hook shape has a part of size $k+1$, and $j$ parts of size $1$.} $s_{(k|j)}(\mbf{x})$, where $d=k+j+1$, for which we can easily show $h$-positivity of $\Theta_{a,b}((-1)^j s_{(k|j)}(\mbf{x}))$, which implies Schur-positivity. Indeed, one easily verifies that the symmetric function   $(-1)^j s_{(k|j)}(\mbf{x})$ expands with positive integer coefficients in the basis $(-1)^{|\mu|-\ell(\mu)}\,e_\mu(\mbf{x})$;
  and we have seen that $\Theta_{a,b}((-1)^{j-1} e_j(\mbf{x}))$ expands with positive integer coefficients in the basis $h_\nu(\mbf{x})$. Hence, applying the homomorphism $\Theta_{a,b}$ to $(-1)^j s_{(k|j)}(\mbf{x})$ gives rise to an $h$-positive expression. This expression is also Schur-positive, since any $h_\nu(\mbf{x})$ is.

Extensive experiments suggest that, for all $\mu$, $\Theta_{a,b}((-1)^{\iota(\mu)} s_\mu(\mbf{x}))$ is Schur-positive, where 
$\iota(\mu)$ is the number of cells $(i,j)$ of the diagram of $\mu$, such that $j>i$. Moreover, all of this seems to carry over to the bi-Frobenius case. An intriguing question is to expand the elliptic Hall algebra techniques to cover these bi-Frobenius. The hope is that this would lead to more explicit formulas for three parameter expressions such as
\begin{equation}\label{aire_dinv}
   \bleu{\sum_{\alpha\in\Dyck{m}{n}} q^{\area(\alpha)} \alpha(\mbf{x};t)\alpha'(\mbf{y};r)},
\end{equation}
 were $\alpha(\mbf{x};t)$ is an LLT-polynomial calculated using the dinv-statistic on $(m,n)$-parking functions (see~\cite{hall} for more details on all this):
    $$\bleu{\alpha(\mbf{x};t):=\sum_{\pi\Park_\alpha} t^{\dinv(\pi)}\,s_{\co(\pi)}(\mbf{x})},$$
 with $\co(\pi)$ is a composition that encodes ``descents'' of the parking function $\pi$. Recall that composition index Schur function may be defined by a suitable adaptation of the Jacobi-Trudi identity. Up to asign-twist,
 expression~\pref{aire_dinv} specializes to the right-hand side of \pref{biFrob}. It is known that the LLT-polynomial $\alpha(\mbf{x};t)$ is Schur-positive.


\end{document}